\documentclass[12pt]{article}
\usepackage{amsmath,amssymb,amsbsy,amsfonts,amsthm,latexsym,amsopn,amstext,
                                               amsxtra,euscript,amscd}
\begin{document}

\newfont{\teneufm}{eufm10}
\newfont{\seveneufm}{eufm7}
\newfont{\fiveeufm}{eufm5}
%
%
\newfam\eufmfam
                \textfont\eufmfam=\teneufm \scriptfont\eufmfam=\seveneufm
                \scriptscriptfont\eufmfam=\fiveeufm
%
%
\def\frak#1{{\fam\eufmfam\relax#1}}
%


\def\bbbr{{\rm I\!R}} 
\def\bbbm{{\rm I\!M}}
\def\bbbn{{\rm I\!N}} 
\def\bbbf{{\rm I\!F}}
\def\bbbh{{\rm I\!H}}
\def\bbbk{{\rm I\!K}}
\def\bbbp{{\rm I\!P}}
\def\bbbone{{\mathchoice {\rm 1\mskip-4mu l} {\rm 1\mskip-4mu l}
{\rm 1\mskip-4.5mu l} {\rm 1\mskip-5mu l}}}
\def\bbbc{{\mathchoice {\setbox0=\hbox{$\displaystyle\rm C$}\hbox{\hbox
to0pt{\kern0.4\wd0\vrule height0.9\ht0\hss}\box0}}
{\setbox0=\hbox{$\textstyle\rm C$}\hbox{\hbox
to0pt{\kern0.4\wd0\vrule height0.9\ht0\hss}\box0}}
{\setbox0=\hbox{$\scriptstyle\rm C$}\hbox{\hbox
to0pt{\kern0.4\wd0\vrule height0.9\ht0\hss}\box0}}
{\setbox0=\hbox{$\scriptscriptstyle\rm C$}\hbox{\hbox
to0pt{\kern0.4\wd0\vrule height0.9\ht0\hss}\box0}}}}
\def\bbbq{{\mathchoice {\setbox0=\hbox{$\displaystyle\rm
Q$}\hbox{\raise 0.15\ht0\hbox to0pt{\kern0.4\wd0\vrule
height0.8\ht0\hss}\box0}} {\setbox0=\hbox{$\textstyle\rm
Q$}\hbox{\raise 0.15\ht0\hbox to0pt{\kern0.4\wd0\vrule
height0.8\ht0\hss}\box0}} {\setbox0=\hbox{$\scriptstyle\rm
Q$}\hbox{\raise 0.15\ht0\hbox to0pt{\kern0.4\wd0\vrule
height0.7\ht0\hss}\box0}} {\setbox0=\hbox{$\scriptscriptstyle\rm
Q$}\hbox{\raise 0.15\ht0\hbox to0pt{\kern0.4\wd0\vrule
height0.7\ht0\hss}\box0}}}}
\def\bbbt{{\mathchoice {\setbox0=\hbox{$\displaystyle\rm
T$}\hbox{\hbox to0pt{\kern0.3\wd0\vrule height0.9\ht0\hss}\box0}}
{\setbox0=\hbox{$\textstyle\rm T$}\hbox{\hbox
to0pt{\kern0.3\wd0\vrule height0.9\ht0\hss}\box0}}
{\setbox0=\hbox{$\scriptstyle\rm T$}\hbox{\hbox
to0pt{\kern0.3\wd0\vrule height0.9\ht0\hss}\box0}}
{\setbox0=\hbox{$\scriptscriptstyle\rm T$}\hbox{\hbox
to0pt{\kern0.3\wd0\vrule height0.9\ht0\hss}\box0}}}}
\def\bbbs{{\mathchoice
{\setbox0=\hbox{$\displaystyle     \rm S$}\hbox{\raise0.5\ht0\hbox
to0pt{\kern0.35\wd0\vrule height0.45\ht0\hss}\hbox
to0pt{\kern0.55\wd0\vrule height0.5\ht0\hss}\box0}}
{\setbox0=\hbox{$\textstyle        \rm S$}\hbox{\raise0.5\ht0\hbox
to0pt{\kern0.35\wd0\vrule height0.45\ht0\hss}\hbox
to0pt{\kern0.55\wd0\vrule height0.5\ht0\hss}\box0}}
{\setbox0=\hbox{$\scriptstyle      \rm S$}\hbox{\raise0.5\ht0\hbox
to0pt{\kern0.35\wd0\vrule height0.45\ht0\hss}\raise0.05\ht0\hbox
to0pt{\kern0.5\wd0\vrule height0.45\ht0\hss}\box0}}
{\setbox0=\hbox{$\scriptscriptstyle\rm S$}\hbox{\raise0.5\ht0\hbox
to0pt{\kern0.4\wd0\vrule height0.45\ht0\hss}\raise0.05\ht0\hbox
to0pt{\kern0.55\wd0\vrule height0.45\ht0\hss}\box0}}}}
\def\bbbz{{\mathchoice {\hbox{$\sf\textstyle Z\kern-0.4em Z$}}
{\hbox{$\sf\textstyle Z\kern-0.4em Z$}} {\hbox{$\sf\scriptstyle
Z\kern-0.3em Z$}} {\hbox{$\sf\scriptscriptstyle Z\kern-0.2em
Z$}}}}
\def\ts{\thinspace}

\newtheorem{theorem}{Theorem}
\newtheorem{lemma}[theorem]{Lemma}
\newtheorem{claim}[theorem]{Claim}
\newtheorem{cor}[theorem]{Corollary}
\newtheorem{prop}[theorem]{Proposition}
\newtheorem{definition}{Definition}
\newtheorem{question}[theorem]{Open Question}

\def\squareforqed{\hbox{\rlap{$\sqcap$}$\sqcup$}}
\def\qed{\ifmmode\squareforqed\else{\unskip\nobreak\hfil
\penalty50\hskip1em\null\nobreak\hfil\squareforqed
\parfillskip=0pt\finalhyphendemerits=0\endgraf}\fi}

\def\cA{{\mathcal A}}
\def\cB{{\mathcal B}}
\def\cC{{\mathcal C}}
\def\cD{{\mathcal D}}
\def\cE{{\mathcal E}}
\def\cF{{\mathcal F}}
\def\cG{{\mathcal G}}
\def\cH{{\mathcal H}}
\def\cI{{\mathcal I}}
\def\cJ{{\mathcal J}}
\def\cK{{\mathcal K}}
\def\cL{{\mathcal L}}
\def\cM{{\mathcal M}}
\def\cN{{\mathcal N}}
\def\cO{{\mathcal O}}
\def\cP{{\mathcal P}}
\def\cQ{{\mathcal Q}}
\def\cR{{\mathcal R}}
\def\cS{{\mathcal S}}
\def\cT{{\mathcal T}}
\def\cU{{\mathcal U}}
\def\cV{{\mathcal V}}
\def\cW{{\mathcal W}}
\def\cX{{\mathcal X}}
\def\cY{{\mathcal Y}}
\def\cZ{{\mathcal Z}}

\newcommand{\comm}[1]{\marginpar{%
\vskip-\baselineskip 
\raggedright\footnotesize
\itshape\hrule\smallskip#1\par\smallskip\hrule}}




\newcommand{\ignore}[1]{}

\def\vec#1{\mathbf{#1}}

\def\e{\mathbf{e}}



\def\GL{\mathrm{GL}}

\hyphenation{re-pub-lished}

\def\rank{{\mathrm{rk}\,}}
\def\ad{{\mathrm ad}}

\def\A{\mathbb{A}}
\def\B{\mathbf{B}}
\def \C{\mathbb{C}}
\def \F{\mathbb{F}}
\def \K{\mathbb{K}}
\def \Z{\mathbb{Z}}
\def \P{\mathbb{P}}
\def \R{\mathbb{R}}
\def \Q{\mathbb{Q}}
\def \N{\mathbb{N}}
\def \Z{\mathbb{Z}}

\def \nd{{\, | \hspace{-1.5 mm}/\,}}

\def\Zn{\Z_n}

\def\Fp{\F_p}
\def\Fq{\F_q}
\def \fp{\Fp^*}
\def\\{\cr}
\def\({\left(}
\def\){\right)}
\def\fl#1{\left\lfloor#1\right\rfloor}
\def\rf#1{\left\lceil#1\right\rceil}

\def\SL{\mathrm{SL}}
\def\GL{\mathrm{GL}}

\def\Sing#1{\mathrm {Sing}\,#1}
\def\invp#1{\mbox{\rm {inv}}_p\,#1}
\def\Mq{\cM_{m,n}(\F_1)}
\def\Mnq{\cM_{n}(\F_q)}
\def\Znq{\cZ_{n}(\F_q)}
\def\Gnq{\GL_{n}(\F_q)}
\def\Snq{\SL_{n}(\F_q)}
\def\MZ{\cM_n(\Z)}
\def\vt{\vec{t}}
\def\MT{\cM_n(\cT)}
\def\MR{\cM_n(\cR)}
\def\MS{\cM_n(\cS)}

\def\SL{\mathrm{SL}}

\def\Ln#1{\mbox{\rm {Ln}}\,#1}

\def\ord#1{{\mathrm{ord}}_p\,#1}

\def\epp{\mbox{\bf{e}}_{p-1}}
\def\ep{\mbox{\bf{e}}_p}
\def\em{\mbox{\bf{e}}_{m}}
\def\ed{\mbox{\bf{e}}_{d}}

\def\ii {\iota}

\def\wt#1{\mbox{\rm {wt}}\,#1}

\def\GR#1{{ \langle #1 \rangle_n }}

\def\ab{\{\pm a,\pm b\}}
\def\cd{\{\pm c,\pm d\}}

\def\Bt {\mbox{\rm {Bt}}}

\def\Res#1{\mbox{\rm {Res}}\,#1}

\def\Tr#1{\mbox{\rm {Tr}}\,#1}

\setlength{\textheight}{43pc}
\setlength{\textwidth}{28pc}

\title{Some Additive Combinatorics Problems in  Matrix Rings}

\author{
{\sc Ron Ferguson} \\
{Department of Mathematics, University of Vlora}\\
{Vlora, Albania}\\
{\tt ronf@univlora.edu.al}
\and
{\sc Corneliu Hoffman} \\
{School of Mathematics, University of Birmingham}\\
{Edgbaston Birmingham B15 2TT, United Kingdom}\\
{\tt C.G.Hoffman@bham.ac.uk}
\and
{\sc Florian~Luca} \\
{Instituto de Matem{\'a}ticas, UNAM}\\
{C.P. 58089, Morelia, Michoac{\'a}n, M{\'e}xico} \\
{\tt fluca@matmor.unam.mx}
\and
{\sc Alina~Ostafe}\\
{Institut f\"ur Mathematik, Universit\"at Z\"urich}\\
{Winterthurerstrasse 190 CH-8057, Z\"urich, Switzerland}\\
{\tt alina.ostafe@math.uzh.ch}
\and
{\sc Igor E.~Shparlinski} \\
{Department of Computing, Macquarie University} \\
{Sydney, NSW 2109, Australia} \\
{\tt igor@ics.mq.edu.au}}

\maketitle

\newpage
\begin{abstract}  We study the distribution of singular and
unimodular matrices in sumsets in  matrix rings over finite fields.
We apply these results to estimate the largest prime divisor of the
determinants in sumsets in  matrix rings over the integers.
\end{abstract}



\section{Introduction}

There is a series of recent works where various problems of additive
combinatorics (see~\cite{TaoVu}) have  been considered in the matrix
rings (see~\cite{BouGam,Chang1,Chang2,CHIKR,Helf1,Helf2} for several recent
results and further references in the area).

Here, we consider several more problems of combinatorial flavor in
the set $\Mnq$  of all $n \times n$ matrices over a finite field
$\F_q$ of $q$ elements.

Furthermore, let $\Gnq$,  $\Snq$ and $\Znq$ be the group of
invertible matrices, the group of matrices of determinant $1$ and
the set of singular matrices, respectively, where all matrices are
from $\Mnq$.

We always assume that $n \ge 2$ and in fact some of our results have
no analogues in the scalar case $n=1$.

Given two sets $\cA,\cB \subseteq \Mnq$, we define
\begin{eqnarray*}
N_{n,q}(\cA,\cB) & = & \#\{A+B  \in \Znq\ : \ A \in \cA,\ B \in \cB\},\\
T_{n,q}(\cA,\cB) & = & \#\{A+B   \in \Snq\ : \ A \in \cA,\ B \in \cB\}.
\end{eqnarray*}

We show that if $\cA$ and $\cB$ are sufficiently large, then
$N_{n,q}(\cA,\cB)$ and $T_{n,q}(\cA,\cB)$ are  close to their
expected value $\#\cA \#\cB/q$. We also adapt the method of D.~Hart,
A.~Iosevich and J.~Solymosi~\cite{HaIoSo} to show that pairwise
products of matrices from the  sumset of $\cA, \cB \subseteq \Mnq$
and the sumset of $\cC,\cD \subseteq \Mnq$ generate the whole group
$\Gnq$, provided that
\begin{equation}
\label{eq:4 sets}
\#\cA \#\cB\# \cC\#\cD \ge c(n) q^{4n^2 - 1}
\end{equation}
holds with a sufficiently large constant $c(n)$ depending only on
$n$. In fact, if $n=1$, that is for the scalar case, we obtain a result of the same strength 
of that of D.~Hart, A.~Iosevich and 
J.~Solymosi~\cite[Theorem~1.4]{HaIoSo} (for $d=2$). 

Although the questions we consider are of combinatorial natures, our
proofs are based on some  tools from analytic number theory and
algebraic geometry. In particular, we use  estimates of character
sums along algebraic varieties due to A.~Skorobogatov~\cite{Skor}
(see also~\cite{Fouv,FoKa,Ka,Lau,Luo,ShpSk} and references therein).
This in turn leads us to study the singularity locus as well as
other properties of some algebraic varieties associated with the
determinant.

Finally, we apply our results to estimate the number of prime divisors
of determinants of matrices from some sumsets of matrices over $\Z$. 

Throughout the paper, we always assume that $i$ and $j$ run
through the set $\{1,\ldots, n\}$.
The implied constants in the symbols `$O$',
and `$\ll$' may depend on the dimension $n \ge 2$. We
recall that the notations $U = O(V)$ and $U \ll V$ are all
equivalent to the assertion that the inequality $|U|\le cV$ holds
for some constant $c>0$.

\section{Preliminaries}

\subsection{Determinantal varieties}

Let $\K=\overline{\F}_q$ be the algebraic closure of $\F_q$. We consider $\cZ_n$ to be the affine variety in $\mathbb{A}^{n^2}_\K$ parameterizing singular
matrices of size $n \times n$.  Then $\Znq$ is the set of
$\F_q$-rational points of the variety $\cZ_n$.

\begin{lemma}
\label{lem:Det Var} The variety $\cZ_n$ defined over $\F_q$ is absolutely irreducible
of dimension $n^2-1$.
\end{lemma}

\begin{proof}
Let $X=(X_{ij})$ be an $n\times n$ matrix of $n^2$ variables
$X_{ij}$ over $\K$. Then $\cZ_n$ is the affine variety defined by
the equation $\det X=0$. Since $\det X$ is an irreducible polynomial
over $\K$ because it is linear in each variable, the variety is
irreducible.

The fact that the dimension of the variety $\cZ_n$ is $n^2-1$ is just a direct consequence of the principal ideal theorem.
\end{proof}

Next, let $\Sing(\cZ_n)$ be the singular locus of $\cZ_n$.

\begin{lemma}
\label{lem:Sing Var} The variety
$\Sing(\cZ_n)$ defined over $\F_q$ is absolutely irreducible
of dimension $n^2-4$.
\end{lemma}

\begin{proof}
Let $X=(X_{ij})$ be an $n\times n$ matrix of indeterminates over
$\K$. It follows from~\cite[Theorem~2.6]{BrVe} that the singular
locus of $\cZ_n$ is the affine variety defined by all the
$(n-1)$-minors of the matrix $X$. In~\cite[Proposition~1.1]{BrVe},
it is proved that this variety is irreducible over $\K$ by
identifying the affine space of $n\times n$ matrices with the affine
space of all $\K$-linear maps $f:\K^n\to \K^n$ whose coordinate ring
is just the polynomial ring $\K[\{X_{ij}\}]$. Then $\Sing(\cZ_n)$ is
just the variety of all linear maps of rank $r<n-1$.

The statement on the dimension of $\Sing(\cZ_n)$ follows immediately from~\cite[Theorems~2.1 and~2.5]{BrVe}.
\end{proof}

\subsection{Character sums over varieties}

Given two matrices $U = (u_{ij}), X = (x_{ij}) \in \Mnq$, we define
their scalar products as
$$
U\cdot X = \sum_{i,j=1}^n u_{ij}x_{ij}.
$$

Let $\psi$ be a fixed nonprincipal  additive character of $\F_q$.
For $U = (u_{ij}) \in \Mnq$ we consider the character sums
$$
S(\Znq,U) = \sum_{X \in \Znq} \psi(U\cdot X),
$$
$$
S(\Snq,U) = \sum_{X \in \Snq} \psi(U\cdot X).
$$

\begin{lemma}
\label{lem:Char Sum 1} Uniformly over all
nonzero matrices $U \in \Mnq$, we have
$$
S(\Znq,U) = O\(q^{n^2 - 5/2}\).
$$
\end{lemma}

\begin{proof}  We recall that by Lemma~\ref{lem:Det Var} the variety $\Znq$
is absolutely irreducible.  It now follows immediately from
a combination of~\cite[Theorem~3.2]{Skor}
with~\cite[Lemma~3.6]{Skor} that
$$
S(\Znq,U) = O\(q^{(n^2 - 1 + s)/2}\),
$$
where $s$ is the dimension of $\Sing(\Znq)$ (see, for example, the
estimate of the sums $S_1(\cY_p, - u)$ in the proof
of~\cite[Theorem~5.1]{Skor}). It now remains to apply
Lemma~\ref{lem:Sing Var}.
\end{proof}

We also have a similar estimate for the exponential sum $S(\Snq,U)$.

\begin{lemma}
\label{lem:Char Sum 2} Uniformly over all
nonzero matrices $U \in \Mnq$, we have
$$
S(\Snq,U) = O\(q^{n^2 - 2}\).
$$
\end{lemma}

\begin{proof} 
Without loss of generality, we may assume that $u_{11} \ne 0$ .
Let $\widetilde X$ be the set of all $n(n-1)$ variable $x_{ij}$, $1\le i,j \le n$ with $i\ne 1$.
We then have
$$
\det X = \sum_{j=1}^n x_{1j} F_j(\widetilde X) 
$$
for some polynomials $F_1, \ldots, F_n$  (in fact,  each $F_j$ depends only on 
$(n-1)^2$ variables, of course). 
Then, 
\begin{equation}
\label{eq:Transf}
S(\Snq,U) =  
\sum_{\widetilde X \in \F_q^{n(n-1)}}  
\sum_{\substack {x_{11}, \ldots, x_{1n} \in \Fq\\
x_{11} F_1(\widetilde X) + \cdots + x_{1n} F_n(\widetilde X) = 1}} \psi\(U\cdot X\) , 
\end{equation}
where the outer sum runs over all the $q^{n(n-1)}$ specialisations 
of $\widetilde X$  over $\F_q$.

If $\widetilde X \in \F_q^{n(n-1)}$ is fixed such that 
the linear forms $x_{11} F_1(\widetilde X) + \ldots + x_{1n} F_n(\widetilde X)$ 
and $x_{11} u_{11} + \cdots + x_{1n} u_{1n}$ 
are linearly independent, then for each $z\in\F_q$ the system of two equations 
$$
x_{11} F_1(\widetilde X) + \cdots + x_{1n} F_n(\widetilde X) =1 \qquad 
\text{and} \qquad
x_{11} u_{11} + \cdots + x_{1n} u_{1n} = z
$$
has exactly $q^{n-2}$ solutions in $x_{11}, \ldots, x_{1n} \in \Fq$.
In this case 
$$
\sum_{\substack {x_{11}, \ldots, x_{1n} \in \Fq\\
x_{11} F_1(\widetilde X) + \cdots + x_{1n} F_n(\widetilde X) = 1}} \psi\(U\cdot X\) 
= q^{n-2} \psi\(\sum_{i=2}^n \sum_{j=1}^n  u_{ij} x_{ij}\) 
\sum_{z\in \Fq} \psi\(z\) 
= 0.
$$

For  $\widetilde X \in \F_q^{n(n-1)}$  such that 
the linear forms $x_{11} F_1(\widetilde X) + \cdots + x_{1n} F_n(\widetilde X)$ 
and $x_{11} u_{11} + \cdots + x_{1n} u_{1n}$ 
are linearly dependent, we estimate the inner sum over 
$x_{11}, \ldots, x_{1n} \in \Fq$ trivially as the number 
of solutions to 
$$
x_{11} F_1(\widetilde X) + \cdots + x_{1n} F_n(\widetilde X) = 1, 
\qquad x_{11}, \ldots, x_{1n} \in \Fq,
$$
which  is $O(q^{n-1})$. 
Furthermore, if $x_{11} F_1(\widetilde X) + \cdots + x_{1n} F_n(\widetilde X)$ 
and $x_{11} u_{11} + \cdots + x_{1n} u_{1n}$ 
are linearly dependent, then 
\begin{equation}
\label{eq:Equation}
F_1(\widetilde X)u_{12} = F_2(\widetilde X)u_{11}.
\end{equation}
Since $u_{11} \ne 0$, equation~\eqref{eq:Equation}
has at most $q^{n(n-1)-1}$ solutions ${\widetilde{X}}$. 

Recalling~\eqref{eq:Transf}, we conclude the proof.
\end{proof}

We note that Lemma~\ref{lem:Char Sum 2} can be alternatively derived
from~\cite{Kow}. 

Our next character sum  is a matrix analogue of the classical
Kloosterman sums (see~\cite{IwKow}). Namely, for $H, U, V \in \Mnq$,
we consider the character sum
$$
K(\Gnq,U,V,H) = \sum_{X \in \Gnq} \psi(U\cdot X + V\cdot (H X^{-1})).
$$

\begin{lemma}
\label{lem:Char Sum K}  Uniformly over all matrices $U, V  \in \Mnq$
among which at least one is a nonzero matrix,   and $H \in \Gnq$, we
have
$$
K(\Gnq,U,V,H)  \ll q^{n^2-1/2}.
$$
\end{lemma}

\begin{proof}
For every $\lambda \in \F_q^*$,  the matrix $\lambda X$ runs through
the whole group $\Gnq$, when so does $X$. Therefore,
\begin{eqnarray*}
\lefteqn{K(\Gnq,U,V,H)  =  \frac{1}{q-1} \sum_{\lambda \in \F_q^*}
\sum_{X \in \Gnq} \psi(U\cdot (\lambda X) + V\cdot (\lambda H X^{-1}))} \\
&   &  \qquad \qquad \qquad \quad = \frac{1}{q-1} \sum_{X \in \Gnq}
\sum_{\lambda \in \F_q^*}\psi(\lambda(U\cdot X) +\lambda^{-1}(V\cdot (H X^{-1}))).
\end{eqnarray*}

If both $U\cdot X $ and $V\cdot  (H X^{-1}) $ are nonzero elements
of $\F_q$, then the sum over $\lambda$ is a Kloosterman sum of size
$O(q^{1/2})$ (see~\cite[Theorem~11.11]{IwKow}).

If only one of $U\cdot X $ and $V\cdot  (H X^{-1}) $ is nonzero
element of $\F_q$, then  the sum over $\lambda$ is equal to $-1$.

Finally, if both $U\cdot X = 0$ and $V\cdot  (H X^{-1}) = 0$, then
the sum over  $\lambda$ is equal to $q-1$. However, because at least
one of  $U$ or $V$ is a nonzero matrix, this happens for at most
$q^{n^2-1}$ matrices $X \in \Gnq$ because $H \in \Gnq$. Now, after
some simple calculations, we obtain the desired bound.
\end{proof}

\section{Singular matrices in sumsets}

We show that if for some fixed $\varepsilon > 0$ we have $\#\cA
\#\cB \ge q^{2n^2 - 3 + \varepsilon}$, then
\begin{equation}
\label{eq:asymp N} N_{n,q}(\cA,\cB) = \left(\frac{1}{q} +
o(1)\right)\#\cA \#\cB,
\end{equation}
as $q \to \infty$.

\begin{theorem}
\label{thm:N}
We have
$$
\left| N_{n,q}(\cA,\cB) - \frac{\# \Znq \#\cA \#\cB}{q^{n^2}}\right|
= O\( q^{n^2-5/2} \sqrt{\# \cA\# \cB}\).
$$
\end{theorem}

\begin{proof}  Let $\psi$ be a nontrivial additive character of $\F_q$.
We have
$$
N_{n,q}(\cA,\cB) = \frac{1 }{q^{n^2}} \sum_{X   \in \Znq}\sum_{A
\in \cA} \sum_{B  \in \cB} \sum_{U   \in  \Mnq} \psi\(U\cdot
\(X-A-B\)\),
$$
as the inner sum vanishes unless $x_{ij} = a_{ij} + b_{ij}$ for all
$i,j =1, \ldots, n$, in which case it is equal to  $q^{n^2}$. Here
 we put $A = (a_{ij})$, $B = (b_{ij})$ and $X = (x_{ij})$.

We now change the order of summation by taking the summation over
$U$ outside, and then  separate the term 
$\# \Znq \#\cA \#\cB/q^{n^2}$ corresponding to the 
zero matrix $U = O_n$,  getting
\begin{eqnarray*}
\lefteqn{\left| N_{n,q}(\cA,\cB) - \frac{\# \Znq \#\cA \#\cB}{q^{n^2}}\right|}\\
& &\qquad  = \frac{1 }{q^{n^2}} \sum_{\substack{U  \in  \Mnq\\U \ne
O_n}} S(\Znq,U) \sum_{A   \in \cA} \psi\(-U\cdot A \)\sum_{B   \in
\cB} \psi\(-\-U\cdot B\).
\end{eqnarray*}

By Lemma~\ref{lem:Char Sum 1}, we have
\begin{eqnarray*}
\lefteqn{\left| N_{n,q}(\cA,\cB) - \frac{\# \Znq \#\cA \#\cB}{q^{n^2}}\right|}\\
& &\qquad  \ll q^{-5/2} \sum_{\substack{U   \in  \Mnq\\U \ne O_n}}
\left|\sum_{A    \in \cA} \psi\(-U\cdot A \) \right|
 \left| \sum_{B  \in \cB}
\psi\(-\-U\cdot B\)\right|\\
& &\qquad  \ll q^{-5/2} \sum_{\substack{U   \in  \Mnq\\U \ne O_n}}
\left|\sum_{A    \in \cA} \psi\(U\cdot A \) \right|
 \left| \sum_{B  \in \cB}
\psi\(U\cdot B\)\right|.
\end{eqnarray*}
We now add the term with $U=O_n$ back
and use the Cauchy inequality. This yields
\begin{eqnarray*}
\lefteqn{
\sum_{\substack{U  \in  \Mnq\\U \ne O_n}}
\left|\sum_{A    \in \cA}
\psi\(U\cdot A \)\right|
 \left| \sum_{B  \in \cB}
\psi\(U\cdot B\)\right|}\\
& &\qquad \le \sum_{U  \in  \Mnq}
\left|\sum_{A    \in \cA}
\psi\(U\cdot A \)\right| \left| \sum_{B  \in \cB}
\psi\(U\cdot B\)\right|\\
& &\qquad \le \sqrt{\sum_{U   \in  \Mnq}
\left|\sum_{A    \in \cA}
\psi\(U\cdot A \)\right|^2} \sqrt{\sum_{U  \in  \Mnq}\left| \sum_{B  \in \cB}
\psi\(U\cdot B\)\right|^2}.
\end{eqnarray*}
We now remark that
\begin{eqnarray*}
 \sum_{U   \in  \Mnq}
\left|\sum_{A    \in \cA}
\psi\(U\cdot A \)\right|^2 &  = &q^{n^2} \# \cA,\\
\sum_{U  \in  \Mnq}\left| \sum_{B  \in \cB}
\psi\(U\cdot B\)\right|^2&=& q^{n^2} \# \cB,
\end{eqnarray*}
which are just variants of the  Parseval identity.

Collecting everything, we obtain the result.
 \end{proof}

Using the fact that $\# \Znq = q^{n^2 - 1} + O\(q^{n^2 - 2}\)$,
which follows, from the well-known formula 
$$
\# \Gnq = q^{(n^2-n)/2} \prod_{j=1}^n (q^j-1) = 
q^{n^2} - q^{n^2-1} + O\(q^{n^2 - 2}\)
$$
(see~\cite[Theorem~99]{Dic58}), we see that Theorem~\ref{thm:N}
implies~\eqref{eq:asymp N}. Furthermore, following the argument of
the proof of Theorem~\ref{thm:N},  
but using Lemma~\ref{lem:Char Sum 2} instead of 
Lemma~\ref{lem:Char Sum 1} in the appropriate place,
we obtain the following statement.

\begin{theorem}
\label{thm:T}
We have
$$
\left| T_{n,q}(\cA,\cB) - \frac{\# \Snq \#\cA \#\cB}{q^{n^2}}\right|
= O\( q^{n^2-2} \sqrt{\# \cA\# \cB}\).
$$
\end{theorem}

In particular, we derive from Theorem~\ref{thm:T} that if for some
fixed $\varepsilon > 0$ we have $\#\cA \#\cB \ge q^{2n^2 - 2 +
\varepsilon}$, then
$$
T_{n,q}(\cA,\cB) = \left(\frac{1}{q} + o(1)\right)\#\cA \#\cB,
$$
as $q \to \infty$.

\section{Generating   $\Gnq$ by  sumset products}

Here, we show that if the sets $ \cA, \cB,  \cC, \cD \subseteq \Mnq$
are large enough then the  sumset products
$$
\{(A+B)(C+D) \ : \ A \in \cA, B \in \cB, C \in \cC, D \in \cD\}
$$
generate the whole group $\Gnq$.

In fact, we give an asymptotic formula for $R(\cA, \cB,  \cC, \cD;
H)$, which is the number of solutions to the equation
$$
H = (A+B)(C+D), \qquad A \in \cA, B \in \cB, C \in \cC, D \in \cD.
$$

\begin{theorem}
\label{thm:SumProd}  Uniformly over all
matrices  $H \in \Gnq$, we have
$$
R(\cA, \cB,  \cC, \cD; H) =
\frac{1}{q^{n^2}}\# \cA \# \cB\# \cC\#\cD
+ O\(q^{n^2-1/2} \sqrt{\# \cA \# \cB \# \cC\#\cD}\).
$$
\end{theorem}

\begin{proof}
Clearly $R(\cA, \cB,  \cC, \cD; H)$ is equal to the
number of solutions to the system of equations
$$
A+B = X, \qquad  C+D = HX^{-1},
$$
where  $A \in \cA$, $B \in \cB$, $C \in \cC$, $D \in \cD$ and
$X\in \Gnq$.

Using the orthogonality
property of characters, we now write
\begin{eqnarray*}
\lefteqn{R(\cA, \cB,  \cC, \cD; H)}\\
& & \quad = \sum_{A \in \cA}\sum_{B \in \cB}
\sum_{C \in \cC}\sum_{D \in \cD}  \sum_{X\in \Gnq} \\
& & \qquad \quad  \frac{1}{q^{2n^2}} \sum_{U,V \in \Mnq}
\psi(U\cdot(X - A - B) + V\cdot(HX^{-1} - C - D))\\
& & \quad =   \frac{1}{q^{2n^2}} \sum_{U,V \in \Mnq}
K(\Gnq,U,V,H) \\
& & \qquad \quad
\sum_{A \in \cA} \psi(-U\cdot A)\sum_{B \in \cB}
\psi(-U\cdot B)
\sum_{C \in \cC} \psi(-V\cdot C) \sum_{D \in \cD} \psi(-V\cdot D) .
\end{eqnarray*}
Separating the contribution of  the
zero matrices  $U = V = O_n$, we obtain
\begin{eqnarray*}
\lefteqn{R(\cA, \cB,  \cC, \cD; H) -
\frac{1}{q^{2n^2}} \# \cA \# \cB\# \cC\#\cD\#\Gnq}\\
& & \quad =   \frac{1}{q^{2n^2}} \sum_{\substack{U,V \in \Mnq\\ U\ne
O_n~\mathrm{or}~V\ne O_n}}
K(\Gnq,U,V,H) \\
& & \qquad \quad
\sum_{A \in \cA} \psi(-U\cdot A)\sum_{B \in \cB}
\psi(-U\cdot B)
\sum_{C \in \cC} \psi(-V\cdot C) \sum_{D \in \cD} \psi(-V\cdot D) .
\end{eqnarray*}
Therefore, by Lemma~\ref{lem:Char Sum K}, we have
\begin{eqnarray*}
\lefteqn{\left|R(\cA, \cB,  \cC, \cD; H) -
\frac{1}{q^{2n^2}} \# \cA \# \cB\# \cC\#\cD\#\Gnq\right|}\\
& &\qquad  = \frac{1}{q^{n^2-1/2}} \sum_{\substack{U,V \in \Mnq\\
U\ne O_n~\mathrm{or}~V\ne O_n}} \left|\sum_{A    \in \cA}
\psi\(-U\cdot A \) \right|
\left| \sum_{B  \in \cB}  \psi\(-U\cdot B\)\right|\\
& & \qquad \qquad\qquad\qquad\qquad\qquad \quad
\left|\sum_{C    \in \cC}  \psi\(-V\cdot C \) \right|
\left| \sum_{D  \in \cD}  \psi\(-V\cdot D\)\right|\\
& &\qquad  \le \frac{1}{q^{n^2-1/2}}
\sum_{U,V \in \Mnq}
\left|\sum_{A    \in \cA}  \psi\(U\cdot A \) \right|
\left| \sum_{B  \in \cB}  \psi\(U\cdot B\)\right|\\
& & \qquad \qquad\qquad\qquad\qquad\qquad \quad
\left|\sum_{C    \in \cC}  \psi\(V\cdot C \) \right|
\left| \sum_{D  \in \cD}  \psi\(V\cdot D\)\right|\\
& &\qquad  = \frac{1}{q^{n^2-1/2}}
\sum_{U  \in \Mnq}
\left|\sum_{A    \in \cA}  \psi\(U\cdot A \) \right|
\left| \sum_{B  \in \cB}  \psi\(U\cdot B\)\right|\\
& & \qquad \qquad\qquad\qquad\qquad\qquad \quad
\sum_{V  \in \Mnq} \left|\sum_{C    \in \cC}  \psi\(V\cdot C \) \right|
\left| \sum_{D  \in \cD}  \psi\(V\cdot D\)\right|.
\end{eqnarray*}
We apply the Cauchy inequality to each of the sums over $U$ and $V$
and, as in the proof of  Theorem~\ref{thm:N}, estimate them as
$q^{n^2} \sqrt{\# \cA \# \cB}$ and   $q^{n^2} \sqrt{\# \cC\#\cD}$,
respectively. We thus obtain
\begin{eqnarray*}
\lefteqn{\left|R(\cA, \cB,  \cC, \cD; H) -
\frac{1}{q^{2n^2}} \# \cA \# \cB\# \cC\#\cD\#\Gnq\right|}\\
& &\qquad \qquad \qquad \qquad \qquad
 \ll q^{n^2-1/2} \sqrt{\# \cA \# \cB \# \cC\#\cD}.
\end{eqnarray*}
Using that $\#\Gnq= q^{n^2} + O\(q^{n^2-1}\)$,
we obtain
\begin{eqnarray*}
\lefteqn{\left|R(\cA, \cB,  \cC, \cD; H) -
\frac{1}{q^{n^2}} \# \cA \# \cB\# \cC\#\cD \right|}\\
& &\qquad \qquad
 \ll  q^{-n^2-1}  \# \cA \# \cB\# \cC\#\cD +  q^{n^2-1/2} \sqrt{\# \cA \# \cB \# \cC\#\cD}.
\end{eqnarray*}
Clearly, the first term never dominates and the result now follows.
\end{proof}

We see from Theorem~\ref{thm:SumProd}
that  if for some fixed $\varepsilon > 0$ we have
$$
\#\cA \#\cB\# \cC\#\cD \ge q^{4n^2 - 1 + \varepsilon},
$$
then uniformly over $H \in \Gnq$,
$$
R(\cA, \cB,  \cC, \cD; H) = \left(\frac{1}{q^{n^2}} + o(1)\right) \#
\cA \# \cB\# \cC\#\cD,
$$
as $q \to \infty$. We also see that  $R(\cA, \cB,  \cC, \cD; H)>0$
under the condition~\eqref{eq:4 sets} with some appropriate constant
$c(n)$.

\section{Prime divisors of sumset determinants}

Given a set $\cT$ of integers, we denote by
$$
\MT = \{T = (t_{ij}) \in \MZ~:~ t_{ij} \in \cT,\ 1 \le i,j\le n\}
$$
the set of  all $n \times n$ matrices
with entries from $\cT$.

We now use our previous results to obtain a lower bound on the
number of distinct prime divisors of the product
$$
W(\cR,\cS) = \prod_{A    \in \MR}  \prod_{B  \in \MS}  \det\(A+B\),
$$
where the sets $\cR,\cS \subseteq\{1, \ldots, N\}$ are dense enough
and $N$ is a sufficiently large integer.

Given a prime $p$ and a set $\cS$ of integers,
we denote by $\nu_p(S)$ the number of residue classes
modulo $p$ which  contain at least one element of $\cS$.

We need the following statement which shows that $\nu_p(\cS)$ is
large for sufficiently many primes. It is a simple variant of
several other results of this type (see, for
example,~\cite{CrEls,Gal}).

\begin{lemma}
\label{lem:ResClass} Let $N$ and $Q$ be sufficiently large positive
integers. Let $\cT \subseteq\{1, \ldots, N\}$ be of cardinality
$\#\cT = T$. If $Q \le T \log N$, then for at least $0.6 Q/\log Q$
primes $p \in [Q,2Q]$ we have
$$
\nu_p(\cT) \ge \frac{p}{20 \log N \log p}.
$$
\end{lemma}

\begin{proof}
Let $\mu_{u,p}(\cT)$ be the number of $t \in \cT$
with $t \equiv u \pmod p$.

By the Cauchy inequality, we have
\begin{equation}
\label{eq:Cauchy} T = \sum_{u =0}^{p-1} \mu_{u,p}(\cT) =
\sum_{\substack{u =0\\ \mu_{u,p}(\cT) \ne 0}}^{p-1} \mu_{u,p}(\cT)
\le \sqrt{\nu_p(S) \sigma_p(\cT)},
\end{equation}
where
$$
\sigma_p(\cT) = \sum_{u =0}^{p-1} \mu_{u,p}(\cT)^2.
$$

We now consider the product
$$
W = \prod_{\substack{t_1,t_2 \in \cT\\ t_1 \ne t_2}} \(t_1-t_2\).
$$
Clearly,
\begin{equation}
\label{eq:Upper W}
1 \le |W| \le N^{T(T-1)}.
\end{equation}
Let $\ord z$ denote the exponent of the prime $p$ in the
factorization of the integer $z$. Collecting together pairs $s,t$ in
the same residue class $u$ modulo $p$, we see that
$$
\ord W \ge \sum_{u =0}^{p-1} \mu_{u,p}(\cT)\(\mu_{u,p}(\cT)-1\) =
\sigma_p(\cT)  - T.
$$
Therefore, using the fact that the number of primes $p\in [Q,2Q]$ is
at most $2Q/\log Q$ for large values of $Q$, which follows from the Prime
Number Theorem, we get that
\begin{equation}
\label{eq:p-adic}
|W| =  \sum_{p} p^{\ord W}
\ge  \prod_{p \in [Q,2Q]} Q^{\ord W} =
 Q^{-2QT/\log Q} \prod_{p \in [Q,2Q]} Q^{\sigma_p(\cT)},
\end{equation}
that provided $Q$ is large enough.

Comparing~\eqref{eq:Upper W} with~\eqref{eq:p-adic}
and recalling that $Q \le T \log N$,
we obtain
\begin{equation}
\label{eq:main ineq}
\sum_{p \in [Q,2Q]} \sigma_p(\cT)
 \le T^2 \log N + 2QT
 \le 3 T^2 \log N.
\end{equation}
Thus,
$$
\# \left\{p \in [Q, 2Q]~:~\sigma_p(\cT) \ge \frac{10 T^2 \log N \log Q}{Q}\right\}
\le  \frac{3Q}{10 \log Q}.
$$
For the remaining primes $p \in [Q,2Q]$, the number of which, by the
Prime Number Theorem, is at least
$$
\(1 - \frac{3}{10} + o(1) \) \frac{Q}{\log Q} \ge  0.6   \frac{Q}{\log Q}
$$
for large enough $Q$, we derive from~\eqref{eq:Cauchy} that
$$
\nu_p(\cT) \ge \frac{Q}{10 \log N \log Q}\ge \frac{p}{20\log p} ,
$$
which concludes the proof. 
\end{proof}

For the purpose of the next result, for a nonzero integer $m$ we write 
$\omega(m)$ for the number of its distinct prime factors.

\begin{theorem}
There exists a positive constant $c_0(n)$ depending only on $n$ such
that if $\cA,~\cB$ are subsets of $\{1,\ldots,N\}$ with
$$
\min\{\#\cA,\#\cB\}>c_0(n)(\log N)^{2n^2/3-1} (\log \log N)^{2n^2/3}
$$
and $N$ is sufficiently large, then
$$
\omega(W(\cA,\cB))\gg\min\{\#\cA,\#\cB\}.
$$
\end{theorem}

\begin{proof} We apply  Lemma~\ref{lem:ResClass} with
$$
T = \min\{\#\cA,\#\cB\} \qquad \text{and} \qquad Q = \fl{T \log N},
$$
getting that for large $Q$ there are are at least
$$
\(0.2 + o(1) \) \frac{Q}{\log Q} \ge  0.1   \frac{Q}{\log Q}
$$
primes $p \in [Q,2Q]$ for which both inequalities
$$
\nu_p(\cA) \ge \frac{p}{10 \log N \log Q}
\qquad \text{and} \qquad
\nu_p(\cB) \ge \frac{p}{10 \log N \log Q}
$$
hold. Since obviously $T \le N$, we have
$$
  \frac{Q}{\log Q} \gg T
$$
such primes. From Theorem~\ref{thm:N}, we see that $$ p \mid
W(\cA,\cB)
$$
provided that
$$
(10 \log N \log Q)^{2n^2}  \le c_1(n) Q^{3}
$$
for an appropriate positive constant $c_1(n)$ depending only $n$.
The above inequality  is satisfied if
$$
(\log N \log T)^{2n^2/3}  \le c_2(n) T \log N
$$
for an appropriate constant $c_2(n)$, and this in turn is implied by
the condition of the theorem with a sufficiently large $c_0(n)$.
\end{proof}

\section*{Acknowledgements}

The authors are grateful to Tony Shaska for the invitation to
NATO Advanced Study Institute ``New Challenges in Digital
Communications'', Vlora, 2008. The stimulating atmosphere of
this meeting, enhanced by Albanian food and wine
has led to  the idea of this work.

The authors  would  also like to thank  Alexei Skorobogatov
for many valuable discussions and clarifications of some
results of~\cite{Skor}.

During the preparation of this paper,  F.~L. was supported in part
by Grant SEP-CONACyT 79685 and PAPIIT 100508, and I.~S. by ARC Grant
DP0556431.

\end{document}